\documentclass[11pt]{article}

\usepackage[T1]{fontenc}
\usepackage[utf8]{inputenc}
\usepackage{lmodern}
\usepackage{microtype}
\usepackage{amsmath,amssymb,amsthm,mathtools}
\usepackage{booktabs}
\usepackage{enumitem}
\usepackage[numbers,sort&compress]{natbib}
\usepackage[hidelinks]{hyperref}
\usepackage[nameinlink,noabbrev]{cleveref}
\usepackage[a4paper,margin=1in]{geometry}

\newtheorem{theorem}{Theorem}[section]
\newtheorem{lemma}[theorem]{Lemma}

\theoremstyle{definition}

\newcommand{\M}{\mathcal{M}}
\newcommand{\GG}{\mathsf{G}}
\newcommand{\UU}{\mathsf{U}}
\newcommand{\TV}{\mathrm{TV}}
\newcommand{\Prob}{\mathbb{P}}
\newcommand{\E}{\mathbb{E}}

\title{Greedy Uniformity on Trees:\\Exact Obstruction and Near-Uniform Spiders}
\author{John Fairfax-Ball}
\date{October 1, 2026}

\begin{document}
\maketitle

\begin{abstract}
Choose a uniformly random ordering of the vertices of a finite tree and run the usual greedy maximal-independent-set algorithm.  We compare the resulting law on maximal independent sets with the uniform law.  We prove that exact uniformity occurs only for the one-vertex tree and the single edge.  The proof is structural: a diameter endpoint exposes a pendant star, and the remaining one-pendant-leaf case is resolved by a strict injection between exact permutation fibres obtained by swapping the pendant leaf with its support vertex.  Exact uniformity is therefore rigid, but it can be approached closely.  For an explicit mixed-spider family $T_{k,l}$ we count the maximal independent sets and compute the exact probability of every output.  With $l=2^k-k$ the total-variation bias is positive and satisfies
\[
 b(T_{k,2^k-k})=O\!\left(\frac{\sqrt{k}}{4^k}\right)
 =O\!\left(\frac{\sqrt{\log n_k}}{n_k^2}\right),
 \qquad n_k=2^k+k+1.
\]
The theorem package has also been formalised in Lean and registered with Palomar.  These records document machine-checked formal verification and the checked axiom boundary; they are not peer review or a certificate of novelty.
\end{abstract}

\section{Introduction}

A random ordering followed by a greedy scan is one of the simplest ways to construct a maximal independent set.  Given a finite graph, scan the vertices in uniformly random order and accept a vertex exactly when none of its previously accepted neighbours has been accepted.  The output is always a maximal independent set, but different maximal independent sets can have very different numbers of generating orders.  The question studied here is when, for a tree, these permutation fibres happen to have equal size.

Write $\M(T)$ for the maximal independent sets of a finite nonempty tree $T$, let $\GG_T$ be the law of the greedy output, and let $\UU_T$ be uniform on $\M(T)$.  Our first result gives a complete obstruction to exact equality.

\begin{theorem}[Exact obstruction]\label{thm:A}
For every finite nonempty tree $T$,
\[
 \GG_T=\UU_T
 \quad\Longleftrightarrow\quad
 T\cong K_1\ \text{or}\ K_2.
\]
Equivalently, if
\[
 b(T)=d_{\TV}(\GG_T,\UU_T)
 =\frac12\sum_{I\in\M(T)}
 \left|\GG_T(I)-\frac1{|\M(T)|}\right|,
\]
then $b(T)=0$ if and only if $T\cong K_1$ or $K_2$.
\end{theorem}

The obstruction is qualitative, so it is natural to ask whether nontrivial connected trees can nevertheless have very small bias.  They can.  Let $T_{k,l}$ be the tree with a centre $c$, $k$ length-two arms $c-u_i-v_i$, and $l$ further leaves adjacent directly to $c$.

\begin{theorem}[Mixed spiders]\label{thm:B}
Let $k\ge0$ and $l\ge1$, and put $N=2^k$.  The tree $T_{k,l}$ has $N+1$ maximal independent sets.  Its unique centre-containing maximal independent set has greedy probability
\[
 p_c=\frac1N\sum_{j=0}^k\binom{k}{j}\frac1{l+2j+1},
\]
and each particular non-centre maximal independent set containing exactly $j$ of the inner vertices $u_i$ has probability
\[
 q_j=\frac{l+2j}{N(l+2j+1)},
\]
with multiplicity $\binom{k}{j}$.

For $k\ge1$, under the tuning $l=2^k-k$, set $A=2^k+1$, let $J\sim\mathrm{Bin}(k,1/2)$, and put $W=2J-k$.  Then
\[
 b(T_{k,2^k-k})
 =\frac12\left(
 \E\frac{W^2}{A^2(A+W)}
 +\E\frac{|W|}{A(A+W)}
 \right)>0.
\]
Consequently, with $n_k=2^k+k+1$,
\[
 b(T_{k,2^k-k})
 =O\!\left(\frac{\sqrt{k}}{4^k}\right)
 =O\!\left(\frac{\sqrt{\log n_k}}{n_k^2}\right)
 \qquad (k\to\infty).
\]
\end{theorem}

Thus exact flatness of the greedy terminal law is confined to the two trivial trees, while an explicit connected sequence approaches the flat law with the upper bound $O(\sqrt{\log n}/n^2)$ along its vertex counts.  We do not prove an all-$n$ upper bound, a matching lower bound, an extremal characterization, or optimality of the mixed-spider family.

The stochastic process itself is classical, as are random-priority and random-sequential-adsorption viewpoints.  A targeted literature audit described in \cref{sec:related} found no checked source containing either \cref{thm:A} or the formulas and tuning in \cref{thm:B}; because such searches cannot certify worldwide priority, we use the deliberately bounded assessment that the contribution-level results are \emph{plausibly new with bounded uncertainty}.

\section{Model and permutation fibres}\label{sec:model}

Let $G=(V,E)$ be a finite simple graph and let $\pi=(v_1,\dots,v_n)$ be an ordering of $V$.  Starting from the empty set, scan $v_1,\dots,v_n$ and accept $v_i$ if it has no accepted neighbour.  Denote the output by $\mathrm{Greedy}(G,\pi)$.  Every rejected vertex has an accepted neighbour at the moment it is scanned, so the final set is independent and dominates every vertex outside it; hence it is maximal independent.

For $I\in\M(G)$ define its \emph{permutation fibre}
\[
 \mathcal F_G(I)=\{\pi:\mathrm{Greedy}(G,\pi)=I\},
 \qquad g_G(I)=|\mathcal F_G(I)|.
\]
Since all $|V|!$ orders are equally likely,
\[
 \GG_G(I)=\frac{g_G(I)}{|V|!}.
\]
Therefore $\GG_G$ is uniform on $\M(G)$ exactly when $g_G(I)$ is constant over $I\in\M(G)$.

We shall use a finite order certificate.  Write $u\prec_\pi v$ when $u$ appears before $v$ in $\pi$.

\begin{lemma}[Order certificate]\label{lem:certificate}
Let $I\in\M(G)$.  Then $\mathrm{Greedy}(G,\pi)=I$ if and only if every $w\notin I$ has a neighbour $u\in I$ with $u\prec_\pi w$.
\end{lemma}

\begin{proof}
If the output is $I$, every $w\notin I$ was rejected by a previously accepted neighbour, which necessarily belongs to $I$.

Conversely, assume the stated condition and suppose that some vertex outside $I$ is accepted.  Choose the earliest accepted such vertex $w$.  By hypothesis there is $u\in I\cap N(w)$ with $u\prec_\pi w$.  If $u$ had been rejected, it would have had an earlier accepted neighbour.  Since $I$ is independent, that neighbour would lie outside $I$, contradicting the choice of $w$.  Thus $u$ was accepted before $w$, forcing $w$ to be rejected, a contradiction.  Hence no outside vertex is accepted, and every vertex of $I$ is accepted because $I$ is independent.
\end{proof}

For background and computation it is also useful to condition on the first scanned vertex.  If $I\in\M(G)$ and $n=|V(G)|>0$, then
\begin{equation}\label{eq:first}
 \GG_G(I)=\frac1n\sum_{v\in I}
 \GG_{G-N[v]}(I\setminus\{v\}),
\end{equation}
with the empty-graph probability interpreted as $1$.  The recurrence is elementary: the first vertex must lie in $I$, and after accepting $v$ only the relative order on $G-N[v]$ remains relevant.  We use \eqref{eq:first} only as background and as a computational check, not in the shortest proof of \cref{thm:A}.

\section{Related work and novelty boundary}\label{sec:related}

Random-order greedy maximal independent sets have been studied in several forms.  Krivelevich, M\'esz\'aros, Michaeli and Shikhelman study the same uniformly random vertex-order process and its equivalent continuous arrival-time formulation, with emphasis on output size, density and local limits \cite{KrivelevichEtAl2020}.  Gazmuri's random greedy deletion algorithm is another equivalent finite implementation; Dall'Asta, Pin and Ramezanpour discuss it in a statistical-mechanics treatment of maximal independent sets and explicitly distinguish a dynamical measure from a flat blocked-state ensemble \cite{Gazmuri1984,DallAstaEtAl2009}.  Blocking random sequential adsorption and parking models supply further background on greedy occupation processes on graphs and trees \cite{Pippenger1989,DehlingEtAl2008,Sudbury2009,PenroseSudbury2005}.

For fixed deterministic graphs, the closest exact-uniformity source located in the final audit is Kryven, Versendaal and de Vries \cite{KryvenEtAl2026}.  Their iterative maximal independent set process chooses uniformly from currently available vertices and is therefore equivalent in law, on a fixed finite graph, to restricting a uniformly random permutation as vertices are removed.  Their Proposition~3.8 gives a \emph{sufficient} structural condition---regular independent sets---under which the terminal maximal independent set is uniform.  The checked version does not state the converse.  Their stronger 2-uniformity condition and its classification therefore do not classify all graphs with exactly uniform greedy output.  On connected trees their sufficient regular-independent-set hypothesis already forces ordinary vertex regularity, leaving only the trivial tree cases, but that one-way implication does not prove the necessity direction of \cref{thm:A}.

Other nearby results concern different observables.  Contat obtains an exact distributional identity for the \emph{size} of the greedy independent set on a uniform Cayley tree \cite{Contat2021}.  The RSA literature largely studies occupation probabilities and limiting densities rather than the complete terminal-set law.  Gadouleau and Kutner study the same deterministic greedy MIS map from an ordering; in particular, their $P_3$ example displays unequal permutation fibres, but their main questions concern reachability and fixing orders rather than random-output equiprobability \cite{GadouleauKutner2025}.  The maximal-independent-set counting inequality used below sits in standard counting territory; compare Sagan and Vatter's $m$-bound \cite{SaganVatter2006}.

These sources make several boundaries important.  We make no novelty claim for the greedy process, its random-priority representation, random sequential adsorption, the dynamic-versus-flat comparison, first-choice recurrences, elementary fibre counting, maximal-independent-set counting bounds, or diameter-end geometry.  The contribution-level claims are the exact tree obstruction in \cref{thm:A} and the exact mixed-spider formulas and tuned near-uniform sequence in \cref{thm:B}.  The literature search was broad but not exhaustive, so throughout we retain the assessment \emph{plausibly new with bounded uncertainty} rather than a priority claim.

\section{The exact obstruction on trees}\label{sec:Aproof}

We first record a simple counting fact.  Let $m(H)=|\M(H)|$, with $m(\varnothing)=1$.

\begin{lemma}\label{lem:count}
For every finite graph $H$ and vertex $z$,
\[
 m(H)\le 2m(H-z).
\]
\end{lemma}

\begin{proof}
Split $\M(H)$ according to whether $z$ is present.  Maximal independent sets excluding $z$ are maximal in $H-z$, so there are at most $m(H-z)$ of them.  Those containing $z$, after deleting $z$, are in bijection with $\M(H-N[z])$.

It remains to show $m(H-N[z])\le m(H-z)$.  For each $J\in\M(H-N[z])$, choose a maximal independent extension $\widehat J$ in $H-z$.  This map is injective: maximality of $J$ in $H-N[z]$ forces
\[
 \widehat J\cap V(H-N[z])=J.
\]
Thus $m(H-N[z])\le m(H-z)$ and the claim follows.
\end{proof}

The tree geometry is equally elementary.

\begin{lemma}[Diameter-end structure]\label{lem:diameter}
Let $T$ have diameter at least $3$ and let $v_0v_1\cdots v_d$ be a diameter path.  Put $y=v_1$ and $z=v_2$.  Every neighbour of $y$ other than $z$ is a leaf.  Hence $y$ has $r\ge1$ leaf neighbours and exactly one nonleaf neighbour $z$.
\end{lemma}

\begin{proof}
The endpoint $v_0$ is a leaf.  If a neighbour $w\ne z$ of $y$ were nonleaf, the component beyond $w$ would contain a leaf $a$ with $\mathrm{dist}(a,y)\ge2$.  The path from $a$ through $y,z,\dots,v_d$ would then have length at least $d+1$, contradicting maximality of the diameter.
\end{proof}

Let $L$ be the set of leaf neighbours of $y$, with $r=|L|$.  Delete $L\cup\{y\}$ and let $R$ be the remaining component containing $z$; put $F=R-z=T-N[y]$.  The maximal independent sets containing $y$ are precisely
\[
 \{y\}\cup A,\qquad A\in\M(F),
\]
so there are $m(F)$.  The maximal independent sets excluding $y$ must contain every leaf in $L$ and are precisely
\[
 L\cup B,\qquad B\in\M(R),
\]
so there are $m(R)$.

\begin{lemma}[Two or more pendant leaves]\label{lem:multileaf}
If $r\ge2$, then the greedy law on $T$ is not uniform on $\M(T)$.
\end{lemma}

\begin{proof}
If $y$ is greedily selected, it must precede every leaf in $L$, since any earlier leaf is accepted immediately and blocks $y$.  Hence
\[
 \Prob_{\GG}(y\text{ selected})\le\frac1{r+1}.
\]
When the nonleaf neighbour $z$ is present, the inequality is strict: among the orders in which $y$ is earliest in $L\cup\{y\}$ are orders beginning $z,y$, but $z$ is then accepted first and blocks $y$.

Under the uniform law on maximal independent sets,
\[
 \Prob_{\UU}(y\in I)=\frac{m(F)}{m(F)+m(R)}.
\]
Applying \cref{lem:count} to $(R,z)$ gives $m(R)\le2m(F)$ and hence $\Prob_{\UU}(y\in I)\ge1/3$.  Therefore, if $z$ exists and $r\ge2$,
\[
 \Prob_{\GG}(y\text{ selected})
 <\frac1{r+1}\le\frac13\le\Prob_{\UU}(y\in I).
\]
If there is no nonleaf neighbour, $T$ is a star with centre $y$.  Its only maximal independent sets are $\{y\}$ and $L$, so the uniform marginal of $y$ is $1/2$, while the greedy marginal is $1/(r+1)<1/2$.
\end{proof}

The remaining case is a support vertex with exactly one pendant leaf.  This is where we compare exact permutation fibres.

\begin{lemma}[One pendant leaf]\label{lem:oneleaf}
Assume the diameter is at least $3$ and $r=1$.  Let $x$ be the unique leaf neighbour of $y$, so locally $x-y-z$ and $\deg(y)=2$.  Then the greedy law on $T$ is not uniform on $\M(T)$.
\end{lemma}

\begin{proof}
Let
\[
 R=T-\{x,y\},\qquad F=R-z=T-N[y].
\]
Fix any $A\in\M(F)$ and set $I_y=\{y\}\cup A$.  Define
\[
 I_x=
 \begin{cases}
  \{x\}\cup A, & \text{if $A$ contains a neighbour of $z$ in $R$},\\
  \{x,z\}\cup A, & \text{otherwise}.
 \end{cases}
\]
Both are maximal independent sets.  Indeed, $I_y$ is maximal because $A$ is maximal in $F$ and $y$ dominates $x$ and $z$.  If $A$ dominates $z$, then $A$ is maximal in $R$, so $\{x\}\cup A$ is maximal in $T$; if $A$ does not dominate $z$, then $A\cup\{z\}$ is independent and maximal in $R$, so $\{x,z\}\cup A$ is maximal in $T$.

Let $\sigma$ be the transposition of the vertex labels $x$ and $y$, applied entrywise to an order.  It is a bijection on all vertex orders.  We compare fibres using \cref{lem:certificate}.

Suppose first that $A$ dominates $z$, so $I_x=\{x\}\cup A$.  For every $\pi\in\mathcal F_T(I_x)$, the swapped order $\sigma\pi$ lies in $\mathcal F_T(I_y)$.  To see this, note first that in $\pi$ the outside vertex $y$ can only be certified by its target neighbour $x$ among $\{x\}\cup A$, so $x\prec_\pi y$; after swapping, $y\prec_{\sigma\pi}x$, certifying the outside leaf $x$ for target $I_y$.  The outside vertex $z$ is certified in $\pi$ by some neighbour in $A$ because $x$ is not adjacent to $z$; that same witness remains earlier after the swap.  Every other outside vertex lies in $F\setminus A$ and is certified by a vertex of $A$, all of which are fixed by $\sigma$.  Thus
\[
 \sigma(\mathcal F_T(I_x))\subseteq\mathcal F_T(I_y).
\]
The inclusion is strict.  Consider an order beginning $y,z$, then listing all vertices of $A$, then the remaining vertices.  By \cref{lem:certificate} it lies in $\mathcal F_T(I_y)$.  Its swapped order begins $x,z$; both $x$ and $z$ are then accepted, so it cannot have output $I_x$, which excludes $z$.  Since $\sigma$ is an involution,
\[
 |\mathcal F_T(I_x)|<|\mathcal F_T(I_y)|.
\]

Now suppose that $A$ does not dominate $z$, so $I_x=\{x,z\}\cup A$.  The reverse inclusion holds:
\[
 \sigma(\mathcal F_T(I_y))\subseteq\mathcal F_T(I_x).
\]
Indeed, for an order in the fibre of $I_y$, the outside leaf $x$ is certified only by $y$, so after swapping the target vertex $x$ precedes and certifies the outside vertex $y$.  Every vertex in $F\setminus A$ is certified by a fixed vertex of $A$, while $z$ is now itself in the target.  Again the inclusion is strict.  An order beginning $z,y$, then listing $x$, the vertices of $A$, and the remaining vertices belongs to $\mathcal F_T(I_x)$; after swapping it begins $z,x$, so $z$ is accepted although $z\notin I_y$.  Hence
\[
 |\mathcal F_T(I_y)|<|\mathcal F_T(I_x)|.
\]

In either case two maximal independent sets have different fibre cardinalities, so their greedy probabilities differ.
\end{proof}

\begin{proof}[Proof of \cref{thm:A}]
For $K_1$ there is a single maximal independent set.  For $K_2$ the first of the two vertices is accepted and blocks the other, so the two singleton maximal independent sets each have probability $1/2$.

Now let $|V(T)|\ge3$.  If $T$ has diameter $2$, then it is a star with at least two leaves and \cref{lem:multileaf} applies.  If the diameter is at least $3$, use \cref{lem:diameter}.  When the support vertex $y$ has at least two leaf neighbours, \cref{lem:multileaf} gives nonuniformity; when it has exactly one leaf neighbour, \cref{lem:oneleaf} gives two maximal independent sets with unequal permutation-fibre sizes.  Thus no tree on at least three vertices has uniform greedy terminal law.
\end{proof}

\section{Mixed spiders and exact output probabilities}\label{sec:spider}

Fix $k\ge0$ and $l\ge1$.  The mixed spider $T_{k,l}$ has centre $c$, length-two arms $c-u_i-v_i$ for $1\le i\le k$, and direct leaves $w_1,\dots,w_l$ adjacent to $c$.

A maximal independent set containing $c$ is forced to contain every outer endpoint $v_i$, and contains no $u_i$ or direct leaf.  Thus there is exactly one centre-containing maximal independent set.  If $c$ is absent, every direct leaf must be present and each arm contributes exactly one of $u_i,v_i$.  Hence there are $2^k$ non-centre maximal independent sets and
\[
 |\M(T_{k,l})|=2^k+1.
\]

Put $N=2^k$.  Temporarily ignore $c$.  On each arm the earlier of $u_i$ and $v_i$ is selected, and the relative orders on the $k$ disjoint pairs are equiprobable.  Thus any prescribed arm pattern has probability $1/N$.

Fix a non-centre pattern containing exactly $j$ inner vertices $u_i$.  Conditional on the required pair orientations, reinserting $c$ destroys that pattern precisely when $c$ is selected.  This happens exactly when $c$ precedes every direct leaf and, on each of the $j$ inner-selected arms, precedes the corresponding $u_i$.  Under the conditioning $u_i\prec v_i$, this is equivalent to $c$ being first among the $l+2j+1$ vertices consisting of $c$, all $l$ direct leaves, and both endpoints of those $j$ arms.  Therefore
\[
 \Prob(c\text{ is selected}\mid\text{prescribed arm pattern})
 =\frac1{l+2j+1}.
\]
It follows that each particular type-$j$ non-centre maximal independent set has probability
\begin{equation}\label{eq:qj}
 q_j=\frac1N\left(1-\frac1{l+2j+1}\right)
 =\frac{l+2j}{N(l+2j+1)}.
\end{equation}
There are $\binom{k}{j}$ such patterns.  Summing the complementary centre-selected contributions over all arm patterns gives
\begin{equation}\label{eq:pc}
 p_c=\frac1N\sum_{j=0}^k\binom{k}{j}\frac1{l+2j+1}.
\end{equation}
This proves the counting and exact-probability parts of \cref{thm:B}.

\section{The tuned bias and its asymptotics}\label{sec:tuned}

Now assume $k\ge1$.  Set $N=2^k$, choose
\[
 l=N-k,
\]
and put $A=N+1$.  Since $k\ge1$, the elementary inequality $k<2^k$ shows that $l>0$.  Let $J\sim\mathrm{Bin}(k,1/2)$ and $W=2J-k$.  The uniform mass on each maximal independent set is $1/A$.

For a type-$J$ non-centre output, \eqref{eq:qj} becomes
\[
 q_J=\frac{N+W}{N(A+W)},
\]
and direct subtraction gives
\begin{equation}\label{eq:qdiff}
 q_J-\frac1A
 =\frac{W}{NA(A+W)}.
\end{equation}
Meanwhile, \eqref{eq:pc} is exactly
\[
 p_c=\E\frac1{A+W}.
\]
Since $\E W=0$,
\begin{align}
 p_c-\frac1A
 &= -\E\frac{W}{A(A+W)} \\
 &= \E\frac{W^2}{A^2(A+W)}.\label{eq:pdiff}
\end{align}
The second identity follows by adding the zero term $\E(W/A^2)$.  Its right-hand side is positive because $A+W>0$ and $W$ is nonzero with positive probability.

Using the multiplicities $\binom{k}{j}$ to convert the non-centre sum to expectation under $J$, \eqref{eq:qdiff} and \eqref{eq:pdiff} yield the exact identity
\begin{equation}\label{eq:biasidentity}
 b(T_{k,N-k})
 =\frac12\left(
 \E\frac{W^2}{A^2(A+W)}
 +\E\frac{|W|}{A(A+W)}
 \right)>0.
\end{equation}
This is the exact expectation identity in the frozen theorem package.

For the upper bound, $W\ge-k$, so $A+W\ge A-k$, while
\[
 \E W^2=k,
 \qquad
 \E|W|\le\sqrt{\E W^2}=\sqrt{k}.
\]
Consequently
\begin{equation}\label{eq:biasbound}
 b(T_{k,N-k})
 \le\frac12\left(
 \frac{k}{A^2(A-k)}
 +\frac{\sqrt{k}}{A(A-k)}
 \right)
 =O\!\left(\frac{\sqrt{k}}{4^k}\right).
\end{equation}
Finally, the vertex count is
\[
 n_k=N+k+1.
\]
For $k\ge1$, $N\le n_k\le2N$, while $k=\log_2N=O(\log n_k)$.  Substituting these comparisons into \eqref{eq:biasbound} gives
\[
 b(T_{k,2^k-k})
 =O\!\left(\frac{\sqrt{\log n_k}}{n_k^2}\right),
\]
completing \cref{thm:B}.

\section{Computational checks}\label{sec:computation}

Computation played a supporting role in theorem discovery and verification, but not in the all-tree proof.  Exact rational arithmetic was used in two independent ways on small trees: direct enumeration of every vertex permutation, and the first-vertex recurrence \eqref{eq:first}.  These methods agree on every nonisomorphic tree through seven vertices.  A broader exact census generated all $436$ nonisomorphic trees on one through eleven vertices (matching OEIS A000055 \cite{OEISA000055}); among them, exactly $K_1$ and $K_2$ have zero bias.

For the tuned mixed-spider family, some exact values are shown in \cref{tab:spider}.  These are checks of the closed formulas, not inputs to their proof.

\begin{table}[ht]
\centering
\caption{Selected exact biases for $T_{k,2^k-k}$.}\label{tab:spider}
\begin{tabular}{rrrr}
\toprule
$k$ & $l$ & $|V|$ & $b(T)$ \\
\midrule
2 & 2 & 7 & $1/30$ \\
3 & 5 & 12 & $7/576$ \\
4 & 12 & 21 & $41/13260$ \\
6 & 58 & 71 & $18353/78602160$ \\
10 & 1014 & 1035 & $22598757953479/19219192053894201600$ \\
\bottomrule
\end{tabular}
\end{table}

A targeted check of the structural proof was also run on all $434$ nonisomorphic trees of orders three through eleven.  In the one-pendant-leaf branch, $1103$ paired maximal-independent-set comparisons had the strict fibre-probability direction predicted by \cref{lem:oneleaf}.  These finite checks are corroborative only; \cref{sec:Aproof} is the proof for arbitrary finite trees.

\section{Lean formalisation and Palomar registration}\label{sec:formal}

The frozen A+B theorem package has been formalised in Lean.  The Stage-5 verification baseline used Lean~4.34.1 and Mathlib commit
\begin{center}
\nolinkurl{d13f23b723b8a846827a245b89c10fc7d3f11612}.
\end{center}
The final Theorem~A declarations are
\begin{center}
\nolinkurl{GreedyUniformity.tree_greedyLawEqUniform_iff_isK1OrK2}\\
\nolinkurl{GreedyUniformity.tree_bias_eq_zero_iff_isK1OrK2}.
\end{center}
The mixed-spider package contains the maximal-independent-set count, the centre and non-centre probability formulas, the tuned exact bias identity and positivity theorem, and the two asymptotic Big-O statements stated above.

For Palomar packaging, Stage~6 migrated the project to Lean~4.35.0-rc3 and Mathlib commit
\begin{center}
\nolinkurl{c55e6e786f49471c72fbddbec5415808896aec1e}.
\end{center}
This was a compatibility migration and did not strengthen or alter the frozen theorem statements.  The package was registered as
\[
 \texttt{PALOMAR-2026-10-01-000003},\qquad \text{version }1,
\]
with trust level \emph{high}.  The immutable registered source commit is
\begin{center}
\nolinkurl{1733c29a5165148d71a2f0bd1ed2dcd7c35309ce}.
\end{center}
The public registry entry is
\begin{center}
\url{https://palomar-registry.org/entry.html?id=PALOMAR-2026-10-01-000003&version=1}.
\end{center}
The project source is available at\begin{center}
\url{https://github.com/jfairfaxball-348/Greedy-Uniformity-on-Trees}.
\end{center}
All nine intended final theorem declarations---two for Theorem~A and seven mixed-spider checkpoints---are included in the registered comparator surface.  The permitted axiom reports for these declarations are exactly
\[
 [\texttt{propext},\ \texttt{Classical.choice},\ \texttt{Quot.sound}],
\]
and there are no project-specific mathematical axioms in the final theorem package.

These records should be interpreted narrowly.  Lean checks the formal derivations against the specified foundations and dependencies, and Palomar records/verifies the packaged formal statements under its tooling.  Neither constitutes peer review, a literature-priority determination, or a certificate that no stronger theorem is known.

\section{Limitations and open directions}\label{sec:open}

The results leave a natural extremal problem.  Define
\[
 a_n=\min\{b(T): T\text{ is a tree on }n\text{ vertices}\}.
\]
Theorem~\ref{thm:A} says only that $a_n>0$ for every $n\ge3$, while Theorem~\ref{thm:B} supplies a small-bias subsequence at the orders $n_k=2^k+k+1$.  We do not determine $a_n$, identify its minimisers, give a matching lower bound, or extend the mixed-spider estimate to every $n$.  The small-order census in \cref{sec:computation} is not evidence for a general minimizer theorem.

It would also be interesting to classify exact uniformity beyond trees.  Kryven--Versendaal--de Vries give a strong sufficient condition \cite{KryvenEtAl2026}, but the general converse problem for the complete random-greedy terminal law is different.  The permutation-fibre viewpoint used in \cref{lem:oneleaf} suggests that local involutions and strict fibre embeddings may be useful in other sparse graph classes.

\section*{Acknowledgements}

The author thanks the developers and maintainers of Lean~4 and Mathlib, and the developers of Palomar's comparator and verification tooling.  AI assistance via ChatGPT was used materially in research planning, literature-audit support, proof development, Lean formalisation, repository maintenance, and manuscript preparation; the author retains responsibility for the mathematical claims, citations, and final text.

\bibliographystyle{plainnat}
\bibliography{references}

\end{document}